\documentclass[11pt,reqno]{amsart}
\usepackage{amsmath, amssymb, amscd, amsfonts, amsthm, mathrsfs, mathtools, mathabx, url, pinlabel, verbatim, flafter, comment, setspace}
\usepackage[text={6in,8in},centering,letterpaper,dvips]{geometry}
\usepackage{color, multirow, dcpic, latexsym, graphicx, epstopdf, subcaption}
\usepackage{subdepth}
\usepackage[colorlinks=true,citecolor=blue,linkcolor=blue]{hyperref}
\usepackage[none]{hyphenat}

\newtheorem {theorem}{Theorem}
\newtheorem {lemma}[theorem]{Lemma}
\newtheorem {proposition}[theorem]{Proposition}
\newtheorem {corollary}[theorem]{Corollary}

\theoremstyle{definition}

\newtheorem *{remark}{Remark}

\newcommand{\tower}{\mathcal{T}^+}

\newcommand{\mfs}{\mathfrak{s}}
\newcommand{\spinc}{\text{Spin}^c}

\newcommand{\HFred}{\mathit{HF}_{red}}
\newcommand{\HFplus}{\mathit{HF}^+}

\newcommand{\hmto}{\widecheck{\mathit{HM}}}
\newcommand{\hmfrom}{\widehat{\mathit{HM}}}
\newcommand{\hmbar}{\overline{\mathit{HM}}}
\newcommand{\hmred}{\mathit{HM}_{red}}
\newcommand{\hsto}{\widecheck{\mathit{HS}}}
\newcommand{\hsfrom}{\widehat{\mathit{HS}}}
\newcommand{\hsbar}{\overline{\mathit{HS}}}
\newcommand{\cfkinf}{\mathit{CFK}^\infty}

\newcommand{\F}{\mathbb{F}}
\newcommand{\Q}{\mathbb{Q}}
\newcommand{\R}{\mathcal{R}}

\title{A remark on the geography problem in Heegaard Floer homology}

\author{Jonathan Hanselman}
\address{Department of Mathematics, Princeton University}
\email{jh66@princeton.edu}

\author{{\c C}a{\u g}atay Kutluhan}
\address{Department of Mathematics, University at Buffalo}
\email{kutluhan@buffalo.edu}

\author{Tye Lidman}
\address{Department of Mathematics, North Carolina State University}
\email{tlid@math.ncsu.edu}

\date{}

\begin{document}
\onehalfspacing
\sloppy
\maketitle

\begin{abstract}We give new obstructions to the module structures arising in Heegaard Floer homology.  As a corollary, we characterize the possible modules arising as the Heegaard Floer homology of an integer homology sphere with one-dimensional reduced Floer homology.  Up to absolute grading shifts, there are only two. We use this corollary to show that the chain complex depicted by Ozsv\'ath, Stipsicz, and Szab\'o in \cite{OSS} to argue that there is no algebraic obstruction to the existence of knots with trivial $\epsilon$ invariant and non-trivial $\Upsilon$ invariant cannot be realized as the knot Floer complex of a knot. \end{abstract}

%%%%%%%%%%%%%%%%%%
\section{Introduction} 
\label{sec:intro}
%!TEX root = geography.tex

Heegaard Floer homology is a collection of three-manifold invariants defined by Ozsv\'ath and Szab\'o which were inspired by the Seiberg--Witten equations in gauge theory \cite{OSInvariance}.  The most refined of these invariants is $\mathit{HF}^+$, which is a graded module over $\F[U]$, where $U$ is an endomorphism of degree $-2$, and $\F$ denotes the field $\mathbb{Z}/2\mathbb{Z}$.  The simplest example is the Heegaard Floer homology of the 3-sphere for which $\mathit{HF}^+(S^3) = \tower_{(0)}$, where $\tower_{(d)}$ denotes $\F[U,U^{-1}]/U \cdot \F[U]$ with $gr(1) = d$.  More interesting examples include $\mathit{HF}^+(\pm \Sigma(2,3,5)) = \tower_{(\mp 2)}$ and $\mathit{HF}^+(\Sigma(2,3,7)) = \tower_{(0)} \oplus \F_{(0)}$, while $\mathit{HF}^+(-\Sigma(2,3,7)) = \tower_{(0)} \oplus \F_{(-1)}$ where a positive orientation refers to the orientation induced on the boundary of a positive-definite plumbing.  In fact, for every $\spinc$ rational homology sphere $(Y,\mfs)$, we have a (non-canonical) splitting $\mathit{HF}^+(Y,\mfs) = \tower_{(d)} \oplus \HFred(Y,\mfs)$, where $\HFred(Y, \mfs)$ is a finitely generated torsion module and $d \in \mathbb{Q}$.  If $Y$ is an integer homology sphere, then there is a unique $\spinc$ structure.  

The $d$-invariant is an invariant of $\spinc$ rational homology cobordism, and has become pervasive in applications to singularity theory, knot concordance, and unknotting numbers of knots (see for instance \cite{BorodzikLivingston, ManolescuOwens, OSUnknotting}).  On the other hand, if $\HFred(Y) = \bigoplus_{\mfs} \HFred(Y,\mfs) = 0$, then $Y$ cannot admit a co-orientable taut foliation \cite{OSGenus}.  The interplay between the $d$-invariants and the reduced Floer homology is also quite powerful; this was used to prove the Dehn surgery characterization of the unknot in $S^3$ \cite{KMOS} (see also \cite{OSGenus, Gainullin}).  

In this note, we give new restrictions on the module structure of the Heegaard Floer homology of rational homology spheres.  

\begin{theorem}\label{thm:main}
Let $Y$ be a rational homology sphere and $\mfs$ a self-conjugate $\spinc$ structure.  If $\HFred(Y,\mfs)$ is supported only in degrees strictly greater than $d(Y,\mfs)$, then $\dim_\F \HFred(Y,\mfs)$ is even.  The same statement holds if $\HFred(Y,\mfs)$ is supported in degrees strictly less than $d(Y,\mfs) - 1$.   
\end{theorem}

Note that every three-manifold admits at least one self-conjugate $\spinc$ structure and the unique $\spinc$ structure on an integer homology sphere is tautologically self-conjugate.  Theorem~\ref{thm:main} immediately yields a characterization of the modules with one-dimensional reduced Floer homology.   

\begin{corollary}
\label{cor:main}
Let $Y$ be a rational homology sphere equipped with a self-conjugate $\spinc$ structure $\mfs$.  If $\dim_\F \HFred(Y,\mfs) = 1$ then $\mathit{HF}^+(Y,\mfs) = \tower_{(d)} \oplus \F_{(d)}$ or $\mathit{HF}^+(Y,\mfs) = \tower_{(d)} \oplus \F_{(d-1)}$.
\end{corollary}

\noindent By the computations of $\mathit{HF}^+(\pm \Sigma(2,3,7))$ stated above, we see that the two possible relatively graded modules with $\dim_\F \HFred = 1$ are realized.   

One may also use the above restrictions to obstruct a chain complex from being realized as the knot Floer complex of a knot inside an integer homology spheres. The following corollary demonstrates this utility of Theorem \ref{thm:main} and Corollary \ref{cor:main}. 
\begin{corollary}
\label{cor:noknotcomplex}
The $\epsilon$-trivial complex $C$ with non-trivial $\Upsilon_C$ depicted in \cite{OSS} cannot be realized as the knot Floer complex of a knot inside an integer homology sphere. 
\end{corollary}

The argument for the proof of Theorem \ref{thm:main} will be a result of the isomorphisms with monopole Floer homology (see Theorem~\ref{thm:klt}) and its relationship, via the Gysin sequence of Lin \cite{Lin1}, with the Pin(2)-monopole Floer homology.  For the reader with a distaste for gauge theory, we point out that the arguments only use the formal properties of these theories.  We briefly review these properties in Section~\ref{sec:background}, and provide a proof of Theorem~\ref{thm:main} in Section~\ref{sec:proofs}. We prove Corollary~\ref{cor:noknotcomplex} in Section \ref{sec:corollary}.

We hope that this note encourages further work utilizing the strengths of both Heegaard Floer homology and Seiberg-Witten theory in conjunction.  

%%%%%%%%%%%%%%%%%%
\subsection*{Acknowledgements}
Jonathan Hanselman was supported in part by NSF grants DMS-1148490 and DMS-1711926. \c{C}a\u{g}atay Kutluhan was supported in part by NSF grant DMS-1360293.  Tye Lidman was supported in part by NSF grants DMS-1148490 and DMS-1709702.
%%%%%%%%%%%%%%%%%%
\section{Background}
\label{sec:background}
%!TEX root = geography.tex

In order to prove Theorem~\ref{thm:main}, we use the Pin(2)-symmetry of solutions of the Seiberg-Witten equations to rule out certain graded module structures in Heegaard Floer homology. To be more precise, we use the Pin(2)-monopole Floer homology as defined by Lin in \cite{Lin1}. The latter is a Morse--Bott version of Kronheimer and Mrowka's monopole Floer homology (see \cite{KM}). In this article, we will not need the definitions of either the monopole Floer homology or the Pin(2)-monopole Floer homology. It suffices to work with their formal properties, which we review next.  

First, to make connection with the Seiberg-Witten equations, we appeal to the isomorphism between Heegaard Floer homology and monopole Floer homology as is proved in \cite{KLT1,KLT2,KLT3,KLT4,KLT5}. 
\begin{theorem}[Main Theorem in \cite{KLT1}]\label{thm:klt}
Let $Y$ be a closed oriented three-manifold and $\mfs$ be a $\spinc$ structure on $Y$.  Then, $\mathit{HF}^+(Y,\mfs)$ (respectively, $\mathit{HF}^-(Y,\mfs)$ and $\mathit{HF}^\infty(Y,\mfs)$) and $\hmto_\ast(Y,\mfs, c_b)$ (respectively, $\hmfrom_\ast(Y,\mfs, c_b)$ and $\hmbar_\ast(Y,\mfs, c_b)$) are isomorphic as relatively graded $\F[U]$-modules.  
\end{theorem} 

\noindent Here, $\mathring{\mathit{HM}}(Y,\mfs,c_b)$ denotes the monopole Floer homology of $(Y,\mfs)$ with a balanced perturbation. In the case of a torsion $\spinc$ structure, this is the same as the standard monopole Floer homology (see \cite[\textsection 30]{KM}). Otherwise, one would work with completions of these modules, denoted $\mathring{\mathit{HM}}_\bullet(Y,\mfs,c_b)$, with respect to the variable $U$, making them modules over $\F[[U]]$. According to \cite[Theorem 31.1.1]{KM}, we have the following isomorphisms:
\[\hmto_\bullet(Y,\mfs,c_b)\cong\hmto_\bullet(Y,\mfs),\;\hmfrom_\bullet(Y,\mfs,c_b)\cong\hmfrom_\bullet(Y,\mfs),\;\hmbar_\bullet(Y,\mfs,c_b)\cong\hmbar_\bullet(Y,\mfs).\] 
In any case, completion does not affect the chain complex that defines $\hmto$ since every element of $\hmto$ is annihilated by some finite power of $U$. 

The part of the isomorphism between Heegaard Floer homology and monopole Floer homology relevant to this article also follows from combined works of Taubes \cite{Taubes1,Taubes2,Taubes3,Taubes4,Taubes5} and Colin, Ghiggini, and Honda \cite{CGH1,CGH2,CGH3}. In fact, together with works of Cristofaro-Gardiner \cite{Cristofaro-Gardiner}, Huang and Ramos \cite{Huang-Ramos}, and Ramos \cite{Ramos}, it follows that for a rational homology sphere $Y$, $\mathit{HF}^+(Y,\mfs)$ and $\hmto_\ast(Y,\mfs)$ are isomorphic as absolutely $\Q$-graded $\F[U]$-modules.

Like in Heegaard Floer and monopole Floer homologies, the Pin(2)-monopole Floer homology comes in a variety of flavors: $\hsto_\bullet$, $\hsfrom_\bullet$, and $\hsbar_\bullet$.  However, the Pin(2)-monopole Floer homology comes equipped with a more interesting module structure; this is the same structure which enables the more refined invariants leading to Manolescu's disproof of the Triangulation Conjecture in dimensions $\geq 5$ \cite{Manolescu} (see also \cite{Lin1}).  For a closed oriented three-manifold $Y$ equipped with a self-conjugate $\spinc$ structure $\mfs$, the invariants $\hsto_\bullet(Y,\mfs), \hsfrom_\bullet(Y,\mfs), \hsbar_\bullet(Y,\mfs)$ take the form of $\mathbb{Q}$-graded modules over $\R = \F[[V]][Q]/Q^3$, where $V$ and $Q$ are endomorphisms of degrees $-4$ and $-1$, respectively.  Note that we can also naturally equip any $\F[[U]]$-module with an $\R$-module structure, by having  $Q$ act by $0$ and $V$ by $U^2$.    

The following proposition displays the clear analogy between the flavors of the monopole Floer and Pin(2)-monopole Floer homologies.  

\begin{proposition}[Proposition 4.6 in Chapter 4 of \cite{Lin1}]\label{thm:hsbar}
Let $\mfs$ be a self-conjugate $\spinc$ structure on a rational homology sphere $Y$.  Then, up to an absolute grading shift, $\hsbar_\bullet(Y,\mfs) \cong \F[[V,V^{-1}][Q]/Q^3$.  
\end{proposition}

\noindent This is analogous to the fact that $\hmbar_\bullet(Y,\mfs) \cong \F[[U,U^{-1}]$ for any $\spinc$ rational homology sphere.  Recall that this implies, using the long exact sequence relating $\hmto_\bullet$, $\hmfrom_\bullet$, and $\hmbar_\bullet$, that in sufficiently large gradings, the dimension of $\hmto_\bullet(Y,\mfs)$ alternates between one and zero.  %This in fact characterizes $\spinc$ rational homology spheres.  In the case when $\mfs$ is self-conjugate, we must have that $\mfs$ is torsion, and thus when $b_1 (Y) \geq 2$, we have $\dim_{\F[[U,U^{-1}]} \hmbar_\bullet(Y,\mfs) \geq 2$ \cite[Theorem 1]{Mark}.  Therefore, in this case, we have $\dim_\F \hmto_k (Y,\mfs) \geq 2$ for all sufficiently large large $k$. Thus, to prove Theorem~\ref{thm:main}, it suffices to consider rational homology spheres. 
Likewise, Proposition~\ref{thm:hsbar} has the following consequence with regard to the rank of $\hsto$ of rational homology spheres in sufficiently large gradings.  

\begin{lemma}\label{lem:high-degree}
Let $\mfs$ be a self-conjugate $\spinc$ structure on a rational homology sphere $Y$.  Then $\dim_\F \hsto_k(Y,\mfs) \leq 1$ for $k \gg 0$.  
\end{lemma} 
\begin{proof}
This follows readily from the definition of the groups $\hsto_\bullet$, $\hsfrom_\bullet$, $\hsbar_\bullet$, the long exact sequence relating them, and Proposition \ref{thm:hsbar}. To be more explicit, by definition $\hsfrom_k(Y,\mfs)$ is zero for all sufficiently large $k\gg0$. Then the long exact sequence,
\[\cdots\xrightarrow{j_\ast}\hsfrom_{k+1}(Y,\mfs)\xrightarrow{p_\ast}\hsbar_{k}(Y,\mfs)\xrightarrow{i_\ast}\hsto_{k}(Y,\mfs)\xrightarrow{j_\ast}\hsfrom_{k}(Y,\mfs)\xrightarrow{p_\ast}\cdots,\] 
implies that $\hsbar_k(Y,\mfs)\cong\hsto_k(Y,\mfs)$ as vector spaces over $\F$ for all sufficiently large $k\gg 0$. On the other hand, Proposition \ref{thm:hsbar} implies that $\hsbar_k(Y,\mfs)$ has rank at most $1$ for any $k\in\mathbb{Z}$. This gives the desired result. 
\end{proof}

\noindent The key fact which allows us to transport information from $\hsto_\bullet$ to $\hmto_\bullet$ is the following Gysin sequence.

\begin{proposition}[Proposition 3.10 in Chapter 4 of \cite{Lin1}]\label{thm:gysin}
Let $Y$ be a closed oriented three-manifold equipped with a self-conjugate $\spinc$ structure $\mfs$.  Then there exists a long exact sequence:
\[
\ldots \to \hsto_{k+1}(Y,\mfs) \overset{e_{k+1}}\to \hsto_{k}(Y,\mfs) \overset{\iota_k}\to \hmto_{k}(Y,\mfs) \overset{\pi_k}\to \hsto_{k}(Y,\mfs) \overset{e_k}\to \hsto_{k-1}(Y,\mfs) \to \ldots 
\]
Further, the maps in this long exact sequence respect the $\R$-module structures.  
\end{proposition}

\noindent With the preceding understood, we are ready to prove Theorem~\ref{thm:main}.  
%%%%%%%%%%%%%%%%%%
\section{Proof of Theorem~\ref{thm:main}}
\label{sec:proofs}
%!TEX root = geography.tex

In order to prove Theorem~\ref{thm:main}, we will simply show that an $\F[[U]]$-module of the form $\tower_{(d)} \oplus N$ where $N$ is an $r$-dimensional torsion module supported in degrees greater than $d$ with $r$ an odd integer cannot fit into the Gysin sequence with an $\R$-module satisfying Lemma~\ref{lem:high-degree}. As explained momentarily, a duality argument rules out the case where $N$ is supported in degrees less than $d-1$. This will imply that such an $\F[[U]]$-module cannot occur as the monopole Floer homology of a rational homology sphere with a self-conjugate $\spinc$ structure.  The $\R$-module structure will be key. 

Note that the isomorphisms of Theorem~\ref{thm:klt} are only relatively graded.  However, from the proof, it will be clear that the absolute grading does not play a role.  We therefore assume for notational simplicity that $d = 0$ throughout.   

Meanwhile, by \cite[Proposition 28.3.4]{KM}, $\hmto_\bullet(Y,\mfs)\cong\hmfrom^\bullet(-Y,\mfs)$ via an isomorphism sending elements in grading $k$ to elements in grading $-(k+1)$. Working with coefficients in the field $\F$, we also have $\hmfrom^\bullet(-Y,\mfs)\cong\hmfrom_\bullet(-Y,\mfs)$. Hence, if $\hmto_\bullet(Y,\mfs)\cong \tower_{(0)} \oplus \hmred(Y,\mfs)$ where $\hmred(Y,\mfs)$ is $r$-dimensional and is supported in degrees less than $-1$ with $r$ an odd integer, then $\hmto_\bullet(-Y,\mfs)\cong \tower_{(0)} \oplus \hmred(-Y,\mfs)$ where $\hmred(-Y,\mfs)$ is $r$-dimensional and is supported in degrees greater than $0$. Therefore, it suffices to prove the non-realizability of $\tower_{(0)} \oplus \hmred$ where $\hmred$ is $r$-dimensional and is supported in degrees greater than $0$ with $r$ an odd integer. 

With the preceding understood, suppose that $\hmto_\bullet(Y,\mfs) \cong\tower_{(0)} \oplus \hmred(Y,\mfs)$ with $\hmred(Y,\mfs)$ supported only in positive degree.
\begin{lemma}\label{lem:low-degree}
 $\hsto_k(Y, \mfs)=0$ for $k < 0$ and $\hsto_0(Y, \mfs) = \F$.
\end{lemma}
\begin{proof} 
Since  $\hmto_k(Y, \mfs)=0$ for $k < 0$, the Gysin sequence in Proposition \ref{thm:gysin} gives isomorphisms $\hsto_k(Y, \mfs) \cong \hsto_{k-1}(Y, \mfs)$ for all $k < 0$. Thus $\hsto_k(Y, \mfs)$ is isomorphic for all $k < 0$. But $\hsto_k(Y, \mfs) = 0$ for sufficiently negative $k$ by definition, so we must have that $\hsto_k(Y, \mfs) = 0$ for all $k < 0$. Finally, $\hmto_0(Y, \mfs) = \F$ (since $\hmred$ is 0 in degree 0), and the exactness of
\begin{equation}\label{eq:pi0}
\cdots \to \hsto_{0}(Y,\mfs) \overset{i_0}\to \hmto_{0}(Y,\mfs) \overset{\pi_0}\to \hsto_{0}(Y,\mfs) \to 0 
\end{equation}
implies that $\pi_0$ is an isomorphism and $\hsto_0(Y, \mfs) = \F$.
\end{proof}

Given $k \ge 0$ even, let $\bar\pi_k$ denote the restriction of the map $\pi_k: \hmto_k(Y, \mfs) \to \hsto_k(Y, \mfs)$ in the Gysin sequence in Proposition \ref{thm:gysin} to the part of the tower $\tower_{(0)}$ with grading $k$.  Although the splitting of $\hmto(Y,\mfs) \cong \tower_{(0)} \oplus \hmred(Y,\mfs)$ is non-canonical, we can identify $\tower_{(0)}$ canonically as a submodule of $\hmto(Y,\mfs)$ by considering the image of $U^\ell$ for $\ell \gg 0$.  Thus, the restriction of $\pi_k$ to $\tower_{(0)}$ is well-defined. 

\begin{lemma}\label{lem:Usquared-action}
For each $i \ge 0$, $\bar\pi_{4i}$ is nontrivial and $\bar\pi_{4i+2}$ is trivial.
\end{lemma}
\begin{proof}
Suppose that $\bar\pi_i$ is nontrivial for some even $i$. We deduce that $\bar\pi_{i+4}$ is also nontrivial from the fact that the Gysin sequence respects the module structures on $\hmto_\bullet$ and $\hsto_\bullet$. In particular $\bar\pi_i \circ U^2 = V \circ \bar\pi_{i+4}$. Since $U^2$ gives a nontrivial map between (the restrictions to $\tower_{(0)}$ of) $\hmto_{i+4}(Y, \mfs)$ and $\hmto_i(Y, \mfs)$ and $\bar\pi_i$ is nontrivial, we must have that $\bar\pi_{i+4}$ is nontrivial. That $\bar\pi_0$ is nontrivial follows from Equation \eqref{eq:pi0} (note that since $\hmred$ is trivial in degree 0, we have $\bar\pi_0 = \pi_0$). By induction, it follows that $\bar\pi_{4i}$ is nontrivial for all $i \ge 0$.

Now suppose that $\bar\pi_{4i+2}$ is also nontrivial for some $i \ge 0$ . Then by the argument above $\bar\pi_{4j + 2}$ is nontrivial for all $j \ge i$, and so $\bar\pi_{2k}$ is nontrivial for all $k \ge 2i$. For sufficiently high degrees (in particular, higher than the support of $\hmred$), the Gysin sequence breaks into pieces of the form
\[
0 \to \hsto_{2k+1}(Y,\mfs) \to \hsto_{2k}(Y,\mfs) \overset{0}\to \F \overset{\bar\pi_{2k}}\hookrightarrow \hsto_{2k}(Y,\mfs) \to \hsto_{2k-1}(Y,\mfs) \to 0. 
\]
Thus we have isomorphisms $\hsto_{2k+1}(Y, \mfs) \cong \hsto_{2k}(Y, \mfs)$ and $\dim_\F \hsto_{2k}(Y, \mfs)$ is strictly larger than $\dim_\F \hsto_{2k-1}(Y, \mfs)$. It follows that $\dim_\F \hsto_{k}(Y, \mfs)$ grows without bound for sufficiently large $k$, violating Lemma \ref{lem:high-degree}. Thus $\bar\pi_{4i+2}$ must be trivial for all $i \ge 0$.
\end{proof}

With these lemmas, we are ready to complete the proof of Theorem~\ref{thm:main}.  For sufficiently large $k$, $\hmto_{4k+1}(Y,\mfs)$ and $\hmto_{4k+3}(Y,\mfs)$ are zero. Consider the portions of the Gysin sequence centered around $\hmto_{4k+2}(Y,\mfs)$ and $\hmto_{4k}(Y,\mfs)$:
\[ 0 \to \hsto_{4k+3}(Y,\mfs) \to \hsto_{4k+2}(Y,\mfs) \twoheadrightarrow \F \overset{0}\to \hsto_{4k+2}(Y,\mfs) \to \hsto_{4k+1}(Y,\mfs) \to 0, \]
\[ 0 \to \hsto_{4k+1}(Y,\mfs) \to \hsto_{4k}(Y,\mfs) \overset{0}\to \F \hookrightarrow \hsto_{4k}(Y,\mfs) \to \hsto_{4k-1}(Y,\mfs) \to 0. \]
Note that the non-triviality (respectively triviality) of the map $\bar\pi_{4k}: \F \to \hsto_{4k}(Y,\mfs)$ (respectively $\bar\pi_{4k+2}: \F \to \hsto_{4k+2}(Y,\mfs)$) is determined by Lemma~\ref{lem:Usquared-action}.
For $k\gg 0$, this gives the following isomorphisms:
\[  \hsto_{4k+1}(Y,\mfs) \cong \hsto_{4k}(Y,\mfs), \]
\[  \hsto_{4k+2}(Y,\mfs) \cong \hsto_{4k+1}(Y,\mfs), \]
\[  \hsto_{4k+3}(Y,\mfs) \oplus \F \cong \hsto_{4k+2}(Y,\mfs),\]
\[  \hsto_{4k+4}(Y,\mfs) \cong \hsto_{4k+3}(Y,\mfs) \oplus \F.\]
In particular, since $\hsto_{4k+4}(Y,\mfs) \cong \hsto_{4k+3}(Y,\mfs) \oplus \F$ and by Lemma \ref{lem:high-degree} we have that $\dim_\F \hsto_{4k+4}(Y,\mfs) \le 1$,  it follows that $\hsto_{4k+3}(Y,\mfs) = 0$ for all sufficiently large $k$.

Fix some large $k$ such that $4k+3$ is larger than the maximum degree in the support of $\hmred(Y,\mfs)$ and $\hsto_{4k+3}(Y,\mfs) = 0$. Consider the Gysin sequence between $\hmto_{4k+3}(Y,\mfs) = 0$ and $\hmto_{-1}(Y,\mfs) = 0$. By exactness, the sum of dimensions of each group in this sequence must be even. The groups that appear in this sequence are $\hmto_i(Y,\mfs)$ for each $0 \le i \le 4k + 2$, $\hsto_{4k+3}(Y,\mfs) = 0$, $\hsto_{-1}(Y,\mfs) = 0$, and two copies of $\hsto_i(Y,\mfs)$ for each $0 \le i \le 4k+2$. It follows that
%\[ \dim \hsto_{4k+3} + \dim \hsto_{-1} + 2 \sum_{i=0}^{4k+2} \dim \hsto_i + \sum_{i=0}^{4k+2} \dim \hmto_i \equiv 0 \quad (\text{mod } 2)\]
\[ \sum_{i=0}^{4k+2} \dim_\F \hmto_i(Y,\mfs) = 2(k+1) + \sum_{i=0}^{4k+2} \dim_\F \mathit{HM}_{red,i}(Y,\mfs) \]
is even, where the first term on the right is the contribution to the tower from degrees zero to $4k+2$, and the second term is simply $\dim_\F \hmred(Y,\mfs)$ since by assumption $\hmred(Y,\mfs)$ is supported in positive degrees at most $4k+2$. Therefore, $\dim_\F \hmred(Y,\mfs)$ must be even.  This completes the proof of Theorem~\ref{thm:main}.
%%%%%%%%%%%%%%%%%%
\section{Proof of Corollary~\ref{cor:noknotcomplex}}
\label{sec:corollary}
%!TEX root = geography.tex
By way of background, Hom introduced a concordance invariant, denoted $\epsilon$, that takes values in the set $\{-1,0,1\}$ \cite{Hom}. This invariant is zero for smoothly slice knots, and it was used by Hom to prove the existence of an infinite rank summand in the subgroup of the smooth concordance group generated by topologically slice knots \cite{HomSummand}. Later on, Ozsv\'ath, Stipsicz, and Szab\'o introduced another concordance invariant, denoted $\Upsilon$, which is a piecewise linear real-valued function on $[0,2]$ \cite{OSS}. Hom showed in \cite{Hom2} that there exist knots $K$ with $\Upsilon_K\equiv0$ but $\epsilon(K)\neq0$. Conversely, Ozsv\'ath, Stipsicz, and Szab\'o presented a chain complex $C$ that is a free $\F[U,U^{-1}]$-module with five generators and differential described by the following diagram:
\begin{figure}[h]
\centering
\includegraphics[width=2.5in]{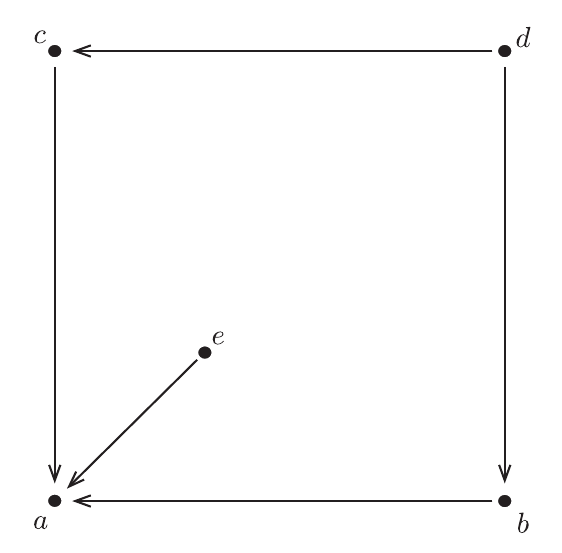}
\caption{The Alexander gradings of the five generators are as follows: $\mathcal{A}(a)=\mathcal{A}(d)=\mathcal{A}(e)=0$, $\mathcal{A}(b)=-3$, and $\mathcal{A}(c)=3$.}
\label{fig:complex}
\end{figure}

\noindent They show in \cite[Proposition 9.4]{OSS} that the chain complex $C$ depicted in Figure \ref{fig:complex} is $\epsilon$-trivial, equivalently $\epsilon(C)=0$, but it has non-trivial $\Upsilon_C$. They use this to argue that there is no algebraic obstruction to existence of knots $K$ with $\epsilon(K)=0$ but $\Upsilon_K\nequiv0$. We show that the chain complex $C$ cannot be realized as the knot Floer complex of a knot.
\begin{proof}
To start the proof, suppose to the contrary that such a knot $K\subset Y$ exists. Then $\cfkinf(Y,K)$ is as in Figure \ref{fig:fullcfk}. 
\begin{figure}[b]
\centering
\begin{subfigure}[b]{0.4\textwidth}
\includegraphics[width=\textwidth]{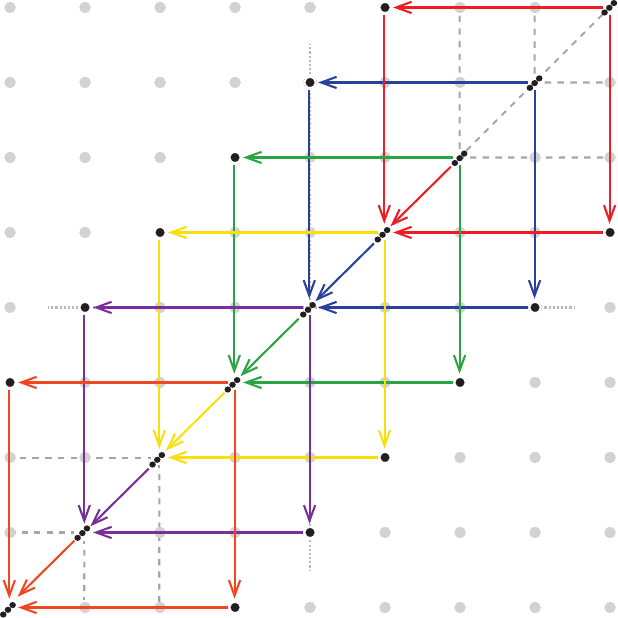}
\caption{}
\label{fig:fullcfk}
\end{subfigure}
\qquad
\begin{subfigure}[b]{0.4\textwidth}
\includegraphics[width=\textwidth]{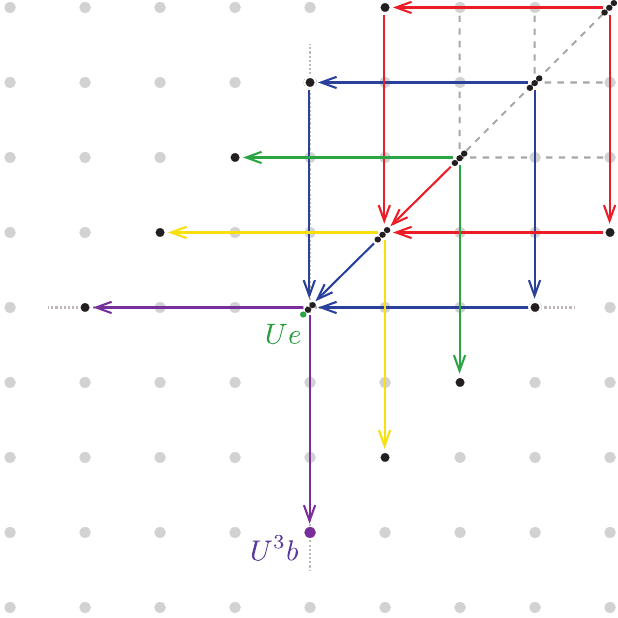}
\caption{}
\label{fig:aplus}
\end{subfigure}
\caption{(\textsc{a}) The complex $\cfkinf(Y,K)$ in the $(i,j)$-plane. Different $U$-translates of the diagram in Figure \ref{fig:complex} are depicted in different colors. (\textsc{b}) The complex $A_0^+$.}
\label{fig:cfk}
\end{figure}
By \cite[Theorem 4.4]{OSHFK} (also see \cite[Theorem 2.3]{OSSurgery}), $\HFplus(Y_p(K),s)\cong H_\ast(A_0^+)$ for large odd values of $p>0$ where $s$ is the unique self-conjugate $\spinc$-structure on $Y_p(K)$ and $A_0^+$ is the chain complex consisting of elements in $\cfkinf(Y,K)$ with $\mathit{max}(i,j)\geq0$ (see Figure \ref{fig:aplus}.) It is easy to check that $H_\ast(A_0^+) = (\F[U,U^{-1}]/U \cdot \F[U])\cdot [U^3b]\oplus \F\cdot[Ue]$. Note also that the relative grading between $Ue$ and $U^3b$ is more than $1$; in fact, $\mathit{gr}(Ue,U^3b)=4$. Therefore, $H_\ast(A_0^+)\cong \tower_{(-4)} \oplus \F_{(0)}$, as relatively-graded groups, which contradicts Corollary \ref{cor:main}. As a result, there does not exist any $K\subset Y$ with $\cfkinf(Y,K)$ as described by the diagram in Figure \ref{fig:complex}.
\end{proof}
\begin{remark}
We would like to note here that if the chain complex $C$ were modified so as to make the diagonal arrow have twice the length then the resulting chain complex would still be $\epsilon$-trivial with non-trivial $\Upsilon$ invariant, but the $\F[U]$-torsion in the homology of the resulting $A_0^+$ complex would have even rank over $\F$. Therefore, we are not able to use the above argument to obstruct that complex from being the knot Floer complex of a knot in an integer homology sphere.
\end{remark}

%%%%%%%%%%%%%%%%%%
\newpage
\bibliographystyle{amsplain}

%\bibliography{references}

\end{document}